\documentclass[10pt]{amsart}
\usepackage{fullpage}
\usepackage{amssymb} 
\usepackage{amsmath} 
\usepackage{amscd}
\usepackage{amsbsy}
\usepackage{comment, enumerate}
\usepackage[matrix,arrow]{xy}
\usepackage[
        colorlinks, citecolor=blue,
        backref,
]{hyperref}
\usepackage{mathrsfs}
\usepackage{color}
\usepackage{mathtools,caption}
\usepackage{todonotes}
\usepackage{stmaryrd}
\usepackage{tikz-cd} 

\newtheorem{theorem}{Theorem}
\newtheorem{lemma}{Lemma}

\newtheorem{proposition}{Proposition}

\newtheorem{question}{Question}

\theoremstyle{remark}
\newtheorem{remark}[theorem]{Remark}

\newcommand{\ZZ}{\mathbb{Z}}
\newcommand{\Zp}{\ZZ_{p}}
\newcommand{\QQ}{\mathbb{Q}}
\newcommand{\CC}{\mathbb{C}}
\newcommand{\HH}{\mathbb{H}}
\newcommand{\PP}{\mathbb{P}}
\newcommand{\Fp}{\mathbb{F}_{p}}
\newcommand{\Fl}{\mathbb{F}_{\ell}}
\newcommand{\cA}{\mathcal{A}}
\newcommand{\cC}{\mathcal{C}}
\newcommand{\cL}{\mathcal{L}}
\newcommand{\cQ}{\mathcal{Q}}
\newcommand{\cR}{\mathcal{R}}
\newcommand{\cV}{\mathcal{V}}
\newcommand{\cX}{\mathcal{X}}
\newcommand{\tV}{\tilde{V}}
\newcommand{\tX}{\tilde{X}}
\newcommand{\Fel}{\mathbb{F}_{11}}

\newcommand{\XDt}{X_{D_{10}}(11)}

\newcommand{\calXDt}{\mathcal{X}_{D_{10}}}
\newcommand{\tilcalXDt}{\tilde{\mathcal{X}}_{D_{10}}}
\newcommand{\XpDt}{X^+_{D_{10}}(11)}

\newcommand{\JDt}{J_{D_{10}}}

\newcommand{\GDt}{\Gamma_{D_{10}}}

\newcommand{\phiXo}{\phi_{X_0}}
\newcommand{\phiXDt}{\phi_{X_{D_{10}}}}

\newcommand{\fp}{\mathfrak{p}}
\newcommand{\Ffp}{\mathbb{F}_{\mathfrak{p}}}

\newcommand{\GL}{\operatorname{GL}}
\newcommand{\PGL}{\operatorname{PGL}}
\newcommand{\SL}{\operatorname{SL}}

\newcommand{\rank}{\operatorname{rank}}
\newcommand{\Tor}{\operatorname{Tor}}
\newcommand{\Gal}{\operatorname{Gal}}
\newcommand{\Tr}{\operatorname{Tr}}

\newcommand{\red}{\operatorname{red}}
\newcommand{\image}{\operatorname{Im}}
\newcommand{\Aut}{\operatorname{Aut}}
\newcommand{\id}{id}

\title{Local-global principle for 11-isogenies of elliptic curves is true over quadratic fields}

\date{\today}

\keywords{Local-global principle, isogenies of elliptic curves, exceptional primes, modular curves, symmetric Chabauty}
\subjclass[2020]{Primary 14G05, 11G18, Secondary 11G05}

\author{Stevan Gajovi\'{c}}
\address{Max Planck Institute for Mathematics in Bonn, Charles University Prague} 
\address{Faculty of Mathematics and Physics,	Charles University,	Ke Karlovu 3, 121 16 Praha 2, Czech Republic}
\email{stevangajovic@gmail.com}

\author{Jeroen Hanselman}
\address{
  AG Algebra, Geometrie und Computeralgebra,
  RPTU Kaiserslautern-Landau
}
\email{hanselm@rptu.de}

\author{Angelos Koutsianas}
\address{Department of Mathematics, Aristotle University of Thessaloniki\\
54124, Thessaloniki, Greece.} 
\email{akoutsianas@math.auth.gr}

\begin{document}

\begin{abstract}
In this paper, we prove that the local-global principle for $11$-isogenies for elliptic curves over quadratic fields holds. This gives a positive answer to a conjecture by Banwait and Cremona \cite[Conjecture 1.14]{BanwaitCremona14}. The proof is based on the determination of the set of quadratic points on the modular curve $\XDt$.
\end{abstract}

\maketitle

\section{Introduction}

Let $K$ be a number field and let $E$ be an elliptic curve over $K$. If $\ell$ is a prime and $E$ admits a $K$-rational $\ell$-isogeny, then it is easy to show that the reduction $\tilde{E}_{\fp}/\Ffp$ of $E$ at a prime $\fp$ of good reduction of $E$ also has a $\Ffp$-rational $\ell$-isogeny, where $\Ffp$ is the residue field of $\fp$. It is natural to ask the converse question:

\begin{question}\label{question:main}
If $\tilde{E}_{\fp}/\Ffp$ admits a $\Ffp$-rational $\ell$-isogeny for a density $1$ set of primes $\fp$, then does $E/K$ admit a $K$-rational $\ell$-isogeny? 
\end{question}

In \cite{Sutherland12} Sutherland studied the above question and explained that in most cases, we can expect the answer to be ``yes''.  However, the most interesting case is when the answer is ``no'' and this local-global property is violated. 

The existence of a $K$-rational $\ell$-isogeny for $E$ depends only on the $j$-invariant $j(E)$ of $E$ when $j(E)\neq 0,1728$; in other words, if $E^\prime/K$ is an elliptic curve with $j(E)=j(E^\prime)$ then the answer is ``no'' for $E/K$ if and only if it is ``no'' for $E^\prime/K$. Following Sutherland, a pair $(\ell, j_0)$ with $j_0\neq 0,1728$ and $j_0\in K$ is called \textit{exceptional for $K$} if there exists an elliptic curve $E/K$ with $j(E)=j_0$ such that the answer to Question \ref{question:main} is ``no''. The prime $\ell$ of an exceptional pair is called an \textit{exceptional prime for $K$}.

We denote by $\rho_{E,\ell}$ the residual Galois representation that arises from the action of $G_K$ on $E(\bar{K})[\ell]$. We denote by $G_{E,\ell}:=\image(\rho_{E,\ell})\subseteq \GL_2(\Fl)$ and $H_{E,\ell}$ is the image of $G_{E,\ell}$ in $\PGL_2(\Fl)$. For $j(E)\neq 0,1728$ the conjugacy class of $H_{E,\ell}$ in $\PGL_2(\Fl)$ depends only on $j(E)$. 

For any number field $K$, it is known that the prime $\ell=2$ is never an exceptional prime by \cite[Remark 2.5]{Anni14}. Therefore, we assume that $\ell$ is odd and set $\ell^{*}=(-1)^{(\ell-1)/2}\ell$. We denote by $D_{2n}$ the dihedral group of order $2n$.

\begin{theorem}[{\cite[Theorem 1, Lemma 1]{Sutherland12}}]\label{thm:Sutherland}
Suppose $\sqrt{\ell^*}\not\in K$ and $(\ell,j_0)$ is an exceptional pair for $K$. Let $E/K$ be an elliptic curve with $j(E)=j_0$. Then, the following statements hold:
\begin{enumerate}
    \item The group $H_{E,\ell}$ is isomorphic to $D_{2n}$, where $n>1$ is an odd divisor of $(\ell-1)/2$.

    \item $\ell\equiv 3\pmod{4}$.

    \item The group $G_{E,\ell}$ is contained in the normaliser of a split Cartan subgroup of $\GL_2(\Fl)$.

    \item $E$ obtains a rational $\ell$-isogeny over $K(\sqrt{\ell^*})$.
\end{enumerate}
\end{theorem}

\begin{remark}\label{rmk:K-rational-points-XD2n}
As a consequence of Theorem~\ref{thm:Sutherland}, when $\sqrt{\ell^*}\notin K$, then any exceptional pair gives rise to a non-cuspidal $K$-rational point on the modular curve $X_{D_{2n}}(\ell)$ (see the definition in Section \ref{sec:X10}) for some odd divisor $n>1$ of $\frac{\ell-1}{2}$.    
\end{remark}

Using Theorem \ref{thm:Sutherland} Sutherland proves that the only exceptional pair for $\QQ$ is $\left(7,\frac{2268945}{128}\right)$ \cite[Theorem 2]{Sutherland12}. In \cite[Proposition 1.3]{BanwaitCremona14} the authors prove an analogous result to Theorem \ref{thm:Sutherland} for the case $\sqrt{\ell^*}\in K$.

\begin{theorem}[{\cite[Proposition 1.3]{BanwaitCremona14}}]\label{thm:Banwait-Cremona}
Suppose $\sqrt{\ell^*}\in K$. Then $(\ell,j_0)$ is exceptional for $K$ if and only if one of the following holds for elliptic curves $E/K$ with $j(E)=j_0$:
\begin{itemize}
    \item $H_{E,\ell}\simeq A_4$ and $\ell\equiv 1\pmod{12}$.

    \item $H_{E,\ell}\simeq S_4$ and $\ell\equiv 1\pmod{24}$.

    \item $H_{E,\ell}\simeq A_5$ and $\ell\equiv 1\pmod{60}$.

    \item $H_{E,\ell}\simeq D_{2n}$ and $\ell\equiv 1\pmod{4}$, $n>1$ is a divisor of $(\ell-1)/2$, and $G_{E,\ell}$ lies in a normaliser of a split Cartan subgroup.
\end{itemize}
\end{theorem}

At the same time, Anni \cite{Anni14} focuses only on exceptional primes and proves that for a given $K$ there are only finitely many.

\begin{theorem}[{\cite[Main Theorem]{Anni14}}]\label{thm:Anni}
Let $K$ be a number field of degree $d$ over $\QQ$ and discriminant $\Delta$, and let $\ell_K:=\max\{|\Delta|, 6d+1\}$. The following holds:
\begin{itemize}
    \item If $(\ell, j_0)$ is an exceptional pair for $K$ then $\ell\leq\ell_K$.
    
    \item There are only finitely many exceptional pairs for $K$ with $7<\ell\leq\ell_K$.
\end{itemize}
\end{theorem}

\begin{remark}
The cases $\ell=2,3,5,7$ are also covered in \cite{Anni14}. In particular, the primes $\ell=2,3$ are not exceptional for any $K$, the prime $\ell=5$ is exceptional if and only if $\sqrt{5}\in K$ and the prime $\ell=7$ appears in infinitely many exceptional pairs for $K$ if and only if the rank over $K$ of the elliptic curve 
\begin{equation*}
    y^2=x^3-1715x+33614,
\end{equation*}
is positive.
\end{remark}

From Theorem \ref{thm:Anni} we understand that there are two important directions in which we can look for exceptional primes. 
\begin{itemize}
    \item Either we fix $K$ and determine all the exceptional primes with $\ell\leq\ell_K$,

    \item or, we fix $\ell$ and a ``suitable'' family $S$ of number fields, and determine all the number fields in $S$ for which $\ell$ is an exceptional prime.
\end{itemize}

\begin{remark}
According to Theorem \ref{thm:Anni} a natural choice of the family $S$ is the set of all number fields of a fixed degree $d$. This choice becomes even more natural because, for odd degree extensions $K$, the bound $\ell_K=6d+1$ is a uniform bound with respect to $d$ \cite[Theorem 4.3]{Anni14}.
\end{remark}

Having all the above results in mind, in particular the fact that $11$ is always in the range of potential exceptional primes when $K$ is a quadratic field, Banwait and Cremona conjectured that the prime $11$ is not an exceptional prime for any quadratic field $K$ \cite[Conjecture 1.14]{BanwaitCremona14}. In this paper, we give a positive answer to the conjecture.

\begin{theorem}\label{thm:main_local_global}
The prime $11$ is not an exceptional prime for any quadratic field.
\end{theorem}

Theorem \ref{thm:main_local_global} is a consequence of Theorem \ref{thm:quadratic_points_XDt} and \cite[Proposition 10.1]{BanwaitCremona14}.

\begin{theorem}\label{thm:quadratic_points_XDt}
The quadratic points of the modular curve $\XDt$ with respect to the model \eqref{eq:XDt} are listed in Table \ref{table:quadratic_points} (up to conjugation).
\end{theorem}

The key ingredients in the proof of Theorem \ref{thm:quadratic_points_XDt} are the development of the (relative) symmetric Chabauty method and the Mordell-Weil sieve \cite{Siksek09b, Box21a, BoxGajovicGoodman22}, the existence of Assaf's implementations of the space of modular forms for an arbitrary congruence subgroup $\Gamma$ \cite{Assaf21}, the LMFDB modular curves database \cite{LMFDB} and the fact that its quotient curve $\XpDt$ is a genus $2$ curve such that $\rank(J_{X_{D_{10}}}(K))=\rank(J_{X^+_{D_{10}}}(K))$ for $K=\QQ,\QQ(\sqrt{-11})$.

\begin{center}
\begin{table}
\begin{tabular}{c  c}
    Name & Coordinates \\ \hline\hline \vspace{1ex}
    $P_1$ &  $(-3/4, 1/4, 0 , \frac{\sqrt{77}}{2}, 0, 1)$ \\
    $P_2$ & $(3/4, -5/4, 0, \frac{\sqrt{77}}{2}, 0, 1)$ \\
    $P_3$ & $(1, 1, 1, \sqrt{-11}, \sqrt{-11}, 1)$ \\
    $P_4$ & $(-1/3, 0, -1/3, \frac{\sqrt{22}}{3}, \frac{\sqrt{22}}{3}, 1)$ \\
    $P_5$ & $(-2/5, 2/5, 1/5, \frac{\sqrt{209}}{5}, -\frac{\sqrt{209}}{5}, 1)$ \\
    $P_6$ & $(-1, 7, 5, \sqrt{473}, -\sqrt{473}, 1)$
\end{tabular}
\caption{The quadratic points of $\XDt$ up to conjugation.}\label{table:quadratic_points}
\end{table}
\end{center}

A number of steps in the proofs were verified computationally using the Magma computer algebra system \cite{Magma}. In our computations we used Magma V2.28-23. We rely on the Modular forms package by Eran Assaf \cite{Assaf21} \url{https://github.com/assaferan/ModFrmGL2} and code by Samir Siksek \cite{Siksekcnf} which is available on \url{https://github.com/samirsiksek/siksek.github.io/tree/main/progs/chabnf}. All of our computations were done on a machine running Ubuntu 22.04.1 with an 2 Intel Xeon E5-2660, 8 core CPUs at 2.2/3.0 GHz, 64 GB RAM. Most computations finish relatively quickly. Only computing the isomorphism between the model computed by Assaf's code and our own simplified equation might take more than an hour. The code used in this paper is available on \url{https://github.com/akoutsianas/local_global_isogenies}. In the paper we will clearly indicate whenever we rely on Magma. Instructions on how to reproduce the steps can be found in the repository.

\section{Acknowledgments}

The second and third authors would like to thank the first author and Charles University for their hospitality when this collaboration took place in Prague. The first and the third author want to thank University of Groningen for their hospitality when they collaborated in Groningen. The first author was supported by Czech Science Foundation GAČR, grant 21-00420M, Junior Fund grant for postdoctoral positions at Charles University and by the MPIM guest postdoctoral fellowship program during various stages of this project. The second author is supported by MaRDI, funded by the Deutsche Forschungsgemeinschaft (DFG), project number 460135501, NFDI 29/1. The third author is supported by Special Account for Research Funds AUTH research grant \textit{``Solving Diophantine Equations-10349''}. Moreover, the third author wants to thank Professor Imin Chen for providing access to the servers of the Mathematics Department of Simon Fraser University where the code was written and tested. We want to thank Steffen M\"uller and Lazar Radi\v{c}evi\'{c} for useful conversations about the project.

\section{The modular curve $\XDt$}\label{sec:X10}

Let $N$ be a positive integer. Suppose $G$ is a subgroup of $\GL_2(\ZZ/N\ZZ)$ and $\Gamma_G=\{A\in\SL_2(\ZZ):A\pmod{N}\in G_0\}$. Because $\Gamma_G \supset \Gamma(N)$ we have that $\Gamma_G$ is a congruence subgroup of $\SL_2(\ZZ)$. We denote by $X_G$ the modular curve that parametrizes elliptic curves over $\CC$ whose residual representation modulo $N$ lies in $G$ up to conjugation. The curve $X_G$ has a model defined over $\QQ(\zeta_N)^{\det(G)}$ where $\det: G\rightarrow (\ZZ/N\ZZ)^*$ is the determinant map. Over $\CC$ the curve $X_G$ is a compact Riemann surface and it holds $X_G(\CC)\simeq\Gamma_G\setminus\HH^*$.


Let $D_{10}\subset\PGL_2(\Fel)$, $G_{D_{10}}$ the pullback of $D_{10}$ to $\GL_2(\Fel)$ and $\Gamma_{D_{10}}:=\Gamma_{G_{D_{10}}}$. We define the modular curve\footnote{The curve $\XDt$ has label \href{https://beta.lmfdb.org/ModularCurve/Q/11.132.6.b.1/}{11.132.6.b.1} in LMFDB beta version \cite{LMFDB}.} $\XDt:=X_{\GDt}$ which is, from the above, defined over $\QQ$ because $\det(G)=\Fel^*$ \cite[Proposition 3]{Sutherland12}.


In \cite{GalbraithThesis} Galbraith described a method to compute a model of a modular curve $X(\Gamma)$ for a congruence subgroup $\Gamma$ as long as we are able to compute a basis of $S_2(\Gamma)$. Galbraith's ideas have been used and extended by Banwait and Cremona\footnote{The authors determine a model for the modular curve $X_{S_4}(13)$.} \cite{BanwaitCremona14, BanwaitThesis}, Zywina \cite{Zywina21} and Box \cite{Box21b}. Moreover, in \cite{Assaf21} Assaf describes a general method of computing $S_2(\Gamma)$ and the algorithm has been implemented in Magma \cite{Magma}.

In our case we use Assaf's algorithm and his implementation to compute $S_2(\GDt)$. We pick an explicit subgroup $H$ of $\PGL_2(\Fel)$ isomorphic to $D_{10}$; in particular, the one that is generated by the following matrices
\begin{equation*}
    A=\begin{pmatrix}
        4 & 0 \\
        0 & 3
    \end{pmatrix},
    \qquad
    B=\begin{pmatrix}
        0 & 1 \\
        1 & 0
    \end{pmatrix}.
\end{equation*}
With the terminology of \cite[Definition 1.2.1]{Assaf21} the group $\GDt$ is of \textit{real type}. Assaf's implementation computes the space $S_2(\GDt)$ which has dimension $6$. This implies that the genus of $\XDt$ is $6$. We note that one can compute the genus of $\XDt$ using known formulas, for example, in \cite[Theorem 3.1.1]{DiamondShurman05}.

\begin{remark} A different choice of the generators of $D_{10}$ gives an isomorphic space of newforms which will not affect our computations of a model of $\XDt$.
\end{remark}

The model we get for $\XDt$ by Assaf's implementation is given by equations with big coefficients. However, we know that $\XDt$ is isomorphic to $X_0(121)$ over $L \coloneqq \QQ(\sqrt{-11})$ \cite[Lemma 3.3.7]{BanwaitThesis}. A model of $X_0(121)$ is given in Galbraith's thesis \cite[p. 36]{GalbraithThesis} which we recall below.

\begin{equation}\label{eq:X0121_model}
X_0(121):~
\begin{cases}
& uw-2vw+2ux-6vx+2uy+2vy+uz = 0,\\[1.5ex]
& uw + vw + 2ux - 2vx + 2uy - 10vy - 5uz + 11vz = 0,\\[1.5ex] 
& -6u^2 + 6uv - 3v^2 -w^2 + 6wx - x^2 - 8wy + 10xy \\
& - 9y^2 - 4wz + 10xz = 0,\\[1.5ex]
& 6u^2 + 12uv + 12v^2 - 17wx - 2x^2 - 5wy + 4xy + 14y^2 \\
& - 6wz - 7xz - 11yz = 0,\\[1.5ex]
& -9v^2 - 8w^2 + wx + 9x^2 + 7wy - 10xy + y^2 - 7wz + 27xz - 11yz = 0,\\[1.5ex]
& -6u^2 - 12uv - 12v^2 - 3w^2 + 7wx - 6x^2 + 11wy + 12xy + 10y^2 + 4wz\\
& +17xz - 11yz - 11z^2=0.
\end{cases}
\end{equation}

We search for a model of $\XDt$ so that it is isomorphic to $X_0(121)$ over $L$ but not over $\QQ$. Simultaneously, we test if it is isomorphic to the model obtained by Assaf's implementation using \textit{IsIsomorphic} in Magma \cite{Magma}. So, at the end we obtain the following model:

\begin{equation}\label{eq:XDt}
\XDt:
\begin{cases}
~& uw - 2vw + 2ux - 6vx + 2uy + 2vy + uz = 0,\\[1.5ex]
& uw + vw + 2ux - 2vx + 2uy - 10vy - 5uz + 11vz = 0,\\[1.5ex]
& - 6u^2 + 6uv - 3v^2 + 11w^2 - 66wx + 11x^2 + 88wy  - 110xy + 99y^2 \\
& + 44wz - 110xz = 0,\\[1.5ex]
& 6u^2 + 12uv + 12v^2 + 187wx + 22x^2 + 55wy - 44xy - 154y^2 + 66wz \\
& + 77xz  + 121yz  = 0,\\[1.5ex]
& - 9v^2 + 88w^2- 11wx -99x^2 - 77wy + 110xy - 11y^2 + 77wz - 297xz \\
& + 121yz    = 0,\\[1.5ex]
& - 6u^2 - 12uv - 12v^2 + 33w^2 - 77wx + 66x^2 - 121wy - 132xy - 110y^2 \\
& - 44wz - 187xz   + 121yz  + 121z^2 = 0
\end{cases}
\end{equation}


The isomorphism $\phi$ over $L$ between the two models of $\XDt$ and $X_0(121)$ is given explicitly, 
\begin{align}\label{eq:phi-def}
\begin{split}
    \phi : ~X_0(121) & \to \XDt,\\
    [x:y:z:u:v:w] & \mapsto [x:y:z:\sqrt{-11}u:\sqrt{-11}v:w].
\end{split}
\end{align}

Let $w_{121}$ be the Atkin-Lehner involution of $X_0(121)$. We define 
\begin{equation}\label{eq:w11_definition}
w_{11}:=\phi\circ w_{121}\circ\phi^{-1}.
\end{equation}
We define $X^+_0(121):=X_0(121)/w_{121}$ and $\XpDt:=\XDt/w_{11}$.

\begin{lemma}\label{lemma:AtkinLehner}
The Atkin-Lehner involution $w_{121}$ on $X_0(121)$ and the involution $w_{11}$ on $\XDt$ are both defined over $\QQ$ and are given by $[x:y:z:u:v:w]\mapsto [x:y:z:-u:-v:w]$ with respect to the models defined by the equations given above.
\end{lemma}

\begin{proof}
We know that $w_{121}$ is defined over $\QQ$. From \cite[Theorem 0.1]{KenkuMomose88}, \cite[Theorem 8]{AtkinLehner70} and \cite[Proposition]{Bars08}  we get that $\Aut_{\QQ}(X_0(121))=\langle w_{121}\rangle\simeq C_2$. Finally, the statement for $w_{11}$ is clear from its definition.
\end{proof}

Then\footnote{We recall that $X^+_0(121)\simeq X_{sp}^+(11)$.} $C:=X^+_0(121)$ has genus $g(C)=2$, hence it is a hyperelliptic curve given by the equation
\begin{equation}\label{eq:X0p}
    C:~y^2 = x^6 - 6x^5 + 11x^4 - 8x^3 + 11x^2 - 6x + 1.
\end{equation}

We also have explicitly computed the quotient maps $\phiXo:~X_0(121) \rightarrow X^+_0(121)$ and $\phiXDt:~\XDt \rightarrow\XpDt$ which can be found in our GitHub repository.

\begin{remark}
In the computations in Section \ref{sec:Mordell_Weil_group} we use the fact that both $\phiXDt$ and $\XpDt$ are defined over $\QQ$.
\end{remark}

\begin{lemma}
The curves $X_0^+(121)$ and $\XpDt$ are isomorphic over $\QQ$.
\end{lemma}

\begin{proof}
We know that $X_0(121)$ and $\XDt$ are isomorphic over $L$ under the isomorphism $\phi$ above. From \eqref{eq:phi-def} and \eqref{eq:w11_definition} we observe that 
\begin{equation*}
\phi^{\sigma}=
\begin{cases}
    \phi, & \sigma\mid_{L}=\id,\\
    \phi\circ w_{121} = w_{11}\circ\phi, & \sigma\mid_{L}\neq\id,
\end{cases}
\end{equation*}
where $\phi^\sigma:=\sigma^{-1}\circ\phi\circ\sigma$ and $\sigma\in \Gal(\bar{\QQ}/\QQ)$. Together with the fact that $\phiXo$ and $\phiXDt$ are defined over $\QQ$ this implies that, when we take quotients by $w_{121}$, and $w_{11}$, the curves $X_0^+(121)$ and $\XpDt$ become isomorphic over $\QQ$.
\end{proof}

By abusing notation, let $\phiXDt$ denote the composition of $\phiXDt$ and the isomorphism between $\XpDt\simeq C$ then we have the commutative diagram
\begin{equation}\label{eq:diagram}
\begin{tikzcd}
X_0(121) \arrow{rr}{\phi}\arrow{rd}{\phiXo} & & \XDt \arrow{ld}{\phiXDt} \\
& C & 
\end{tikzcd}
\end{equation}
where every map has been computed explicitly.

\section{The Mordell-Weil group of $\JDt$}\label{sec:Mordell_Weil_group}

We denote by $J_0$, $\JDt$ and $J_C$ the Jacobians of $X_0(121)$, $\XDt$ and $C$, respectively. It holds that 
\begin{equation}\label{eq:J121_isogeny}
J_0\sim E_{f_1}\oplus E_{f_2}\oplus E_{f_3}\oplus E_{f_4}\oplus E^2_{f_5},
\end{equation}
where $f_i$ for $i=1,2,3,4$ are the four rational newforms of level $121$, the ordering is according to the LMFDB, and $f_5$ is the unique rational newform of level $11$. By modularity, $f_i$ corresponds to the elliptic curve $E_{f_i}$ over $\QQ$ for each $i$. It holds that $\rank_{\QQ}(E_{f_2})=1$ and $\rank_{\QQ}(E_{f_i})=0$ for $i\neq 2$. Therefore, $\rank_{\QQ}(J_0)=1$. 

Moreover, over $L$ we have $\rank_{L}(E_{f_i})=0$ for $i\neq 2$ and $\rank_{L}(E_{f_2})=2$. Because the isogeny in \eqref{eq:J121_isogeny} is defined over $\QQ$, we get $\rank_{L}(J_0)=\rank_{L}(\JDt)=2$.

The curve $C$ has two points at infinity with $\infty=(1,1,0)$ and $(-\infty) = (1,-1,0)$ in the weighted projective plane $\PP^2(1,3,1)$. Using the implementation of the method \cite{Stoll01} in Magma, we prove that $J_C(\QQ)\simeq \ZZ/5\ZZ\times \ZZ$ with generators $G_1 = \left[(0,-1) - (-\infty)\right]$ and $G_2 = \left[\infty - (-\infty)\right]$ of order $5$ and infinite, respectively. Since Chabauty's rank condition is satisfied for $C$ ($\rank(J_C(\QQ))=1<2=g(C)$), we use the implementation of the method of Chabauty and Coleman in Magma to compute $C(\QQ)$.  

\begin{proposition}
We have that $C(\QQ)=\{(1, \pm 2), (0, \pm 1), \pm\infty\}$.
\end{proposition}

For the computations and the proofs that follow it is important to explicitly compute $\phiXDt^{-1}(C(\QQ))$. Using Magma we get
\begin{align*}
    \phiXDt^{-1}((1,-2)) & = \{P_1, \bar{P}_1\},\quad \phiXDt^{-1}((1,2)) = \{P_2, \bar{P}_2\}, \\
    \phiXDt^{-1}(-\infty) & = \{P_3, \bar{P}_3\}, \quad \phiXDt^{-1}(\infty) = \{P_4, \bar{P}_4\},\\
    \phiXDt^{-1}((0,-1)) & = \{P_5, \bar{P}_5\},\quad \phiXDt^{-1}((0,1)) = \{P_6, \bar{P}_6\},
\end{align*}
where $\bar{P}_i$ is the conjugate of $P_i$.

We will use the commutative diagram \eqref{eq:diagram} and the fact that $\XDt$ is isomorphic to $X_0(121)$ over $L$ in order to determine a finite index subgroup of $\JDt(\QQ)$.

The first step is to determine $J_C(L)$. Using the information of $J_C(L)$ and the fact that $\phiXDt$ is defined over $\QQ$ we are able to determine a finite index supgroup of $\JDt(\QQ)$.

\begin{proposition}\label{prop:JCK_mordellweil}
It holds that $J_C(L)_{\Tor}=\langle G_1\rangle$ and $\rank(J_C(L))=2$. In particular, the group generated by $\langle G_1, G_2, G_3\rangle$ is a finite index subgroup of $J_C(L)$ with
\begin{equation*}
    G_3 = \left[\left(\frac{-1 + \sqrt{-3}}{2}, -\sqrt{-11}\right)  + \left(\frac{-1 - \sqrt{-3}}{2}, -\sqrt{-11}\right) - \infty - (-\infty)\right].
\end{equation*}
\end{proposition}

\begin{proof}
By computing the order of the reduction of $J_C$ modulo good primes and using the fact that the torsion subgroup injects into these groups (in Magma) we get that $\# J_C(L)_{\Tor}\leq 5$. Because $\# J_C(\QQ)_{\Tor}=5$ and $J_C(\QQ)_{\Tor}\subset J_C(L)_{\Tor}$ we understand that $J_C(L)_{\Tor}=J_C(\QQ)_{\Tor}=\langle G_1\rangle$.

Using the implementation of 2-descent \cite{Stoll01} in Magma, we also get that $\rank(J_C(L))\leq 2$. The point $G_3$ is a point of infinite order with $G_3\not\in J_C(\QQ)$. We prove that $G_2$ and $G_3$ are linearly independent elements in $J_C(L)$ by using Siksek's code. The idea is the following; if $G_2, G_3$ are linear dependent then the pairs $(\lambda,\mu)\in\ZZ^2$ such that $\lambda G_1+\mu G_2=0$ is a subgroup $\Lambda$ of $\ZZ^2$. We can also assume that there exists a pair $(\lambda,\mu)$ such that $\lambda, \mu$ are coprime away $\#J_C(L)_{\Tor}$. Using the reduction of $J_C(L)$ modulo small primes we  prove that exist a prime $p\nmid \#J_C(L)_{\Tor}$ such that $\Lambda\subseteq p\ZZ^2$.
\end{proof}

\begin{proposition}\label{prop:finite-index-sbgp-of-J(Q)}
It holds that $\rank(\JDt(\QQ))=1$ and $\JDt(\QQ)_{\Tor}$ is isomorphic to $C_5$ or $C_5\times C_5$. In particular, the group $G=\langle D_1, D_2\rangle$ where 
\begin{align*}
    D_1 & = \left[P_6 + \bar{P}_6 - P_4 - \bar{P}_4\right],\\
    D_2 & = \left[P_3 + \bar{P}_3 - P_4 - \bar{P}_4\right],
\end{align*}
with $5D_1=0$ and $D_2$ has infinite order, is a finite index subgroup of $\JDt(\QQ)$ such that $10\JDt(\QQ)\subseteq G$.
\end{proposition}

\begin{proof}
As in Proposition \ref{prop:JCK_mordellweil} we prove that $\#\JDt(\QQ)_{\Tor}\mid 25$. At the same time we also prove that $\JDt(\QQ)_{\Tor}\not\simeq C_{25}$, therefore the torsion part of $\JDt(\QQ)$ is isomorphic to a subgroup of $C_5\times C_5$. We have that $\phiXDt^*$ injects $J_C(\QQ)_{\Tor}$ into $\JDt(\QQ)_{\Tor}$ because $(\deg(\phiXDt), \#\JDt(\QQ)_{\Tor})=1$. Since, $J_C(\QQ)_{\Tor}\simeq C_5$ we understand that $\JDt(\QQ)_{\Tor}\simeq C_5$ or $C_5\times C_5$. Let $D_1=\phiXDt^*(G_1)$ then it holds $5D_1=0$.

The quotient map $\phiXDt\colon \XDt\longrightarrow C$ induces an isogeny $\JDt\sim J_C\times A$ where $A/\QQ$ is some abelian variety and the isogeny is defined over $\QQ$. Since $\rank(\JDt(L))=\rank (J_C(L))=2$, we have $\rank(A(L))=0$, so $\rank(A(\QQ))=0$. Hence, $\rank(\JDt(\QQ))=\rank(J_C(\QQ))=1$. 

Let $D_2 = \phiXDt^*(G_2)\in \JDt(\QQ)$, then $D_2$ has an infinite order in $\JDt(\QQ)$ because $G_2$ has an infinite order in $J_C(\QQ)$. Finally, from the above and  \cite[Proposition 3.1]{Box21a} we have that $10\JDt(\QQ)\subseteq G$.
\end{proof}

\section{Symmetric Chabauty}\label{sec:symmetric_chabauty}

In this section we apply the relative symmetric Chabauty as it is described in \cite{Siksek09b, Box21a, BoxGajovicGoodman22}. We also give a brief exposition of the idea of symmetric Chabauty following \cite[Section 2.2]{Box21a}.

\subsection{Introduction}
Let $X/\QQ$ be a smooth projective curve of genus $g$ that has good reduction at a prime $p>2$. Let $J$ be the Jacobian of $X$ and $r$ the rank of $J(\QQ)$. We also assume that $r\leq g-2$, i.e., the Chabauty condition holds for quadratic points. We denote by $\cX$ the proper minimal regular model of $X$. In the space of global differential forms $\Omega_{X/\QQ_p}(X)$, we have the $\ZZ_p$-submodule $\Omega_{\mathcal{X}/\ZZ_p}(\mathcal{X})$. In \cite{Coleman85} Coleman defines the pairing
\begin{equation*}
    \Omega_{X/\QQ_p}(X)\times J(\QQ_p)\rightarrow \QQ_p, \quad \left(\omega, \left[\sum_i P_i - Q_i\right]\right)\mapsto \sum_{i}\int_{Q_i}^{P_i}\omega.
\end{equation*}
We denote by $V$ the annihilator of $J(\QQ)$ under the above pairing, write  $\cV=V\cap\Omega_{\mathcal{X}/\ZZ_p}(\mathcal{X})$ and let $\tV$ be the image of $\cV$ under the reduction map. 

Let $\cQ=\lbrace Q_1, Q_2\rbrace\in X^{(2)}(\QQ)$, where $X^{(2)}$ is the symmetric square of $X$. Suppose $\omega_1,\cdots, \omega_k$ is a basis of $\tV$. We fix a place $v$ of $\QQ(Q_1, Q_2)$ above $p$ and denote by $\tilde{Q}_i$ the reduction of $Q_i$ with respect to $v$. Let $t_{\tilde{Q_i}}$ be a uniformiser of $\tilde{Q_i}$. Then, we can expand $\omega_j$ around $\tilde{Q}_i$ with respect to $t_{\tilde{Q}_i}$ as a formal power series. In particular, it holds that
\begin{equation}
    \omega_j=(a_0(\omega_j, t_{\tilde{Q_i}}) + a_1(\omega_j, t_{\tilde{Q_i}})t_{\tilde{Q_i}} + a_2(\omega_j, t_{\tilde{Q_i}})t^2_{\tilde{Q_i}} + \cdots)dt_{\tilde{Q_i}}
\end{equation}
If $Q_1\neq Q_2$, we define
\begin{equation*}
    \cA = \begin{pmatrix}
        a_0(\omega_1, t_{\tilde{Q_1}}) & a_0(\omega_1, t_{\tilde{Q_2}})\\
        \vdots & \vdots \\
        a_0(\omega_k, t_{\tilde{Q_1}}) & a_0(\omega_k, t_{\tilde{Q_2}})
    \end{pmatrix},
\end{equation*}
and when $Q_1=Q_2$, we define
\begin{equation*}
    \cA = \begin{pmatrix}
        a_0(\omega_1, t_{\tilde{Q_1}}) & a_1(\omega_1, t_{\tilde{Q_1}})/2\\
        \vdots & \vdots \\
        a_0(\omega_k, t_{\tilde{Q_1}}) & a_1(\omega_k, t_{\tilde{Q_1}})/2
    \end{pmatrix}.
\end{equation*}

\begin{theorem}[{\cite[Theorem 3.2]{Siksek09b}}]\label{thm:single_point_in_residue_class}
Suppose $p\geq 5$. If $\rank(\cA)=2$, then $\cQ$ is the only point on $X^{(2)}(\QQ)$ in its residue class modulo $p$, i.e. for any $\cR\in X^{(2)}(\QQ)$ such that $\cQ\equiv \cR \pmod{v}$, we have $\cR=\cQ$.
\end{theorem}

Suppose $\rho:X\rightarrow C$ is a degree $2$ map to a smooth non-singular curve $C$ and $\cC$ its proper minimal regular model. We also assume that $C$ has good reduction at $p$. Suppose that $\rho$ extends to a morphism $\cX\rightarrow\cC$ which corresponds to a degree $2$ map  $\tX=\tilde{\cX}\rightarrow\cC_{\Fp}$. The map $\rho:X\rightarrow C$ induces the trace map on the holomorphic differentials $\Tr\colon\Omega_{X /\QQ_p}(X)\rightarrow\Omega_{C/\QQ_p}(C)$, for more details, see \cite[\S2.6]{BoxGajovicGoodman22}.

Let $\cV_0=\mathcal{V}\cap\ker\left(\Omega_{X /\QQ_p}(X)\xrightarrow{\Tr}\Omega_{C/\QQ_p}(C)\right)$ and $\tV_{0}$ the image of $\cV_0$ under the reduction map. Let $\cQ=\lbrace Q_1, Q_2\rbrace\in\rho^*C(\QQ)$, $\omega_1,\cdots,\omega_{k'}$ be a basis of $\tV_0$ and $p$, $v$, $\tilde{Q}_1$, $t_{\tilde{Q}_1}$ be as above.

\begin{theorem}[{\cite[Theorem 4.3]{Siksek09b}}]\label{thm:pullback_points_in_residue_class}
Suppose $p\geq 5$. If there exist $\omega_i$ for some $i\in\{1,\cdots, k'\}$ such that 
\begin{equation*}
    \frac{\omega_i}{dt_{\tilde{Q}_1}}\mid_{t_{\tilde{Q}_1}=0}\neq 0,
\end{equation*}
then every point in $X^{(2)}(\QQ)$ in the residue class of $\cQ$ belongs to $\rho^*(C(\QQ))$. 
\end{theorem}

\subsection{Relative Symmetric Chabauty for $\XDt$}
Let $W$ be the image of $1-w_{11}^*:\Omega_{\XDt/\QQ}\rightarrow \Omega_{\XDt/\QQ}$. 

\begin{proposition}\label{prop:annihilator_space}
The space $W$ annihilates $\JDt(\QQ)$ under the integration pairing, hence
\begin{equation*}
    \int_0^{D}\omega = 0,
\end{equation*}
for all $\omega\in W$ and $D\in \JDt(\QQ)$. In addition, $W$ lies in the kernel of the $\Tr_{\phiXDt}:\Omega_{\XDt/\QQ}\rightarrow\Omega_{C/\QQ}$.
\end{proposition}

\begin{proof}
The proof is similar to \cite[Lemma 3.4]{Box21a} because $\rank(\JDt(\QQ))=\rank(J_C(\QQ))$ and $1+w_{11}^*$ is the trace map with respect to $\phiXDt$. We also use \cite[Lemma 2.2]{BoxGajovicGoodman22}.
\end{proof}

Let $\tilde{W}$ be the image of $W\cap\Omega_{\calXDt/\Zp}$ under the reduction map.

\begin{proposition}\label{prop:annihilator_space_modp}
The space $\tilde{W}$ is the image of $1-\tilde{w}_{11}^*:\Omega_{\tilcalXDt/\Fp}\rightarrow\Omega_{\tilcalXDt/\Fp}$.
\end{proposition}

\begin{proof}
We use \cite[Lemma 3.6]{Box21a} and the proof follows in the same way as the proof of \cite[Proposition 3.5]{Box21a}.
\end{proof}

\section{Mordell-Weil sieve}\label{sec:mordell_weil}

In this section, we briefly recall the Mordell-Weil sieve as discussed in \cite{Siksek09b, Box21a, BoxGajovicGoodman22} and which has its origin in \cite{BruinStoll08}. Again we follow \cite[Section 2.4]{Box21a}.

Suppose $X/\QQ$ be smooth projective curve with Jacobian $J$ and $\rho:X\rightarrow C$ a degree 2 map (defined over $\QQ$) where $C/\QQ$ is another curve. We assume that we have the following data:
\begin{enumerate}
    \item $D_1,\cdots, D_r$ are divisors of $J(\QQ)$ that generate a finite index subgroup $G$ of $J(\QQ)$,

    \item A number $N$ such that $NJ(\QQ)\subset G$,

    \item A rational degree $2$ divisor which we denote by $\infty$,
    
    \item $\cL^\prime$ is a known finite subset of $X^{(2)}(\QQ)$. The set $\cL^\prime$ may also include points from $\rho^*(C(\QQ))$,

    \item $p_1,\cdots, p_n$ are primes of good reduction for $X$.
\end{enumerate}
We define $\cL=\cL^\prime\cup\rho^*C(\QQ)$. 

Let $p$ be one of the primes $p_1,\cdots, p_r$. We define the maps $\phi:\ZZ^n\rightarrow G$ where $\phi(a_1,\cdots, a_n)=\sum_{i=1}^{n}a_iD_i$, $\iota:X^{(2)}(\QQ)\rightarrow G$ where $\iota(\cQ)=N[\cQ-\infty]$ and $\iota_p:\tX^{(2)}(\Fp)\rightarrow J(\Fp)$ where $\iota_p(\cR)=N[\cR - \tilde{\infty}]$. In particular, we have the following diagram

\begin{center}
\begin{tikzcd}
\cL \arrow{r}{i} \arrow{rd} & X^{(2)}(\QQ) \arrow{r}{\iota} \arrow{d}{\red} & G \arrow{d}{\red}& \ZZ^r\arrow{l}{\phi} \arrow{ld}{\phi_p}\\

& \tX^{(2)}(\Fp) \arrow{r}{\iota_p}& J(\Fp) & 
\end{tikzcd}
\end{center}
where $\red$ is the reduction map and $\phi_p:=\red\circ~\phi$.

Let $\mathcal{M}_p\subset \tX^{(2)}(\Fp)$ be the subset of points $\cR\in \iota_p^{-1}(\image(\phi_p))$ which satisfy one of the following:
\begin{itemize}
    \item $\cR\not\in \red(\cL^\prime)$,

    \item $\cR=\tilde{\cQ}$ for some point $\cQ\in\cL^\prime\setminus\rho^*C(\QQ)$ not satisfying the conditions of Theorem \ref{thm:single_point_in_residue_class},

    \item $\cR=\tilde{\cQ}$ for some $\rho^*C(\QQ)$ not satisfying the conditions of Theorem \ref{thm:pullback_points_in_residue_class}.
\end{itemize}

\begin{theorem}[Mordell-Weil sieve {\cite[Theorem 5.1]{Siksek09b}}, {\cite[Theorem 2.6]{Box21a}}]\label{thm:mordell_weil_sieve}
Suppose
\begin{equation*}
    \bigcap_{j=1}^{r}\phi^{-1}_{p_j}\left(\iota_{p_j}\left(\mathcal{M}_{p_j}\right)\right)=\emptyset,
\end{equation*}
then $X^{(2)}(\QQ)=\cL$.
\end{theorem}

\section{Proof of Theorem \ref{thm:quadratic_points_XDt}}

In this section we give the proof of Theorem \ref{thm:quadratic_points_XDt}. With the notation of Section \ref{sec:mordell_weil} we set $\cL^\prime=\emptyset$ and $\cL=\phiXDt^*C(\QQ)\subset \XDt^{(2)}(\QQ)$.

\begin{proof}[Proof of Theorem \ref{thm:quadratic_points_XDt}] We apply relative symmetric Chabauty for $p=5,7,13,17,19,23$ and we prove that the elements in $\cL$ are the only elements in their residue disks modulo $p$. Because $\cL^\prime=\emptyset$ we only had to apply Theorem \ref{thm:pullback_points_in_residue_class}. Finally, it is enough to apply the Mordell-Weil sieve, Theorem \ref{thm:mordell_weil_sieve}, where $G$ is the finite-index subgroup of $\JDt(\QQ)$ from Proposition~\ref{prop:finite-index-sbgp-of-J(Q)}, and $N=10$, for the above choice of primes and get $\cL=\XDt^{(2)}(\QQ)$.
\end{proof}

\begin{proof}[Proof of Theorem \ref{thm:main_local_global}]
We may assume that $K\neq \QQ(\sqrt{-11})$ by Theorem~\ref{thm:Banwait-Cremona}. For the other quadratic fields, we may apply Theorem~\ref{thm:Sutherland} and Remark~\ref{rmk:K-rational-points-XD2n}, which explains that it is enough to compute the quadratic points on $\XDt$. 

From Theorem \ref{thm:quadratic_points_XDt} we have determined the quadratic points of $\XDt$. All the quadratic points of $\XDt$ are pullbacks of the rational points of $C$. The image of the set of rational points of $C$ under the $j$-map has been computed in the LMFDB modular curves database\footnote{Since $C\simeq X_{sp}^+(11)$ the $j$-map of $C$ can be found in \url{https://beta.lmfdb.org/ModularCurve/Q/11.66.2.a.1/}.} and is equal to
\begin{equation*}
    J = [\infty, -3375, 8000, -884736, 16581375, -884736000].
\end{equation*}

Let $\Phi_{11}(x,y)$ be the class modular polynomial. With a short script we show that $\Phi_{11}(x, j_0)$ has a linear factor over $\QQ[x]$, hence it also has a linear factor over any quadratic extension $K/\QQ$, for all possible $j_0\in J$ with $j_0\neq\infty$, which by \cite[Proposition 10.1]{BanwaitCremona14} is enough to conclude the proof.
\end{proof}

\bibliographystyle{plain}
\bibliography{my_bibliography}{}

\begin{thebibliography}{10}

\bibitem{Anni14}
Samuele Anni.
\newblock A local-global principle for isogenies of prime degree over number
  fields.
\newblock {\em J. Lond. Math. Soc., II. Ser.}, 89(3):745--761, 2014.

\bibitem{Assaf21}
Eran Assaf.
\newblock Computing classical modular forms for arbitrary congruence subgroups.
\newblock In Jennifer~S. Balakrishnan, Noam Elkies, Brendan Hassett, Bjorn
  Poonen, Andrew~V. Sutherland, and John Voight, editors, {\em Arithmetic
  Geometry, Number Theory, and Computation}, pages 43--104, Cham, 2021.
  Springer International Publishing.

\bibitem{AtkinLehner70}
A.~O.~L. Atkin and J.~Lehner.
\newblock Hecke operators on {$\Gamma \sb{0}(m)$}.
\newblock {\em Math. Ann.}, 185:134--160, 1970.

\bibitem{BanwaitThesis}
Barinder~S. Banwait.
\newblock {\em On some local to global phenomena for abelian varieties}.
\newblock ProQuest LLC, Ann Arbor, MI, 2013.
\newblock Thesis (Ph.D.)--University of Warwick (United Kingdom).

\bibitem{BanwaitCremona14}
Barinder~S. Banwait and John~E. Cremona.
\newblock Tetrahedral elliptic curves and the local-global principle for
  isogenies.
\newblock {\em Algebra Number Theory}, 8(5):1201--1229, 2014.

\bibitem{Bars08}
Francesc Bars.
\newblock The group structure of the normalizer of {$\Gamma_0(N)$} after
  {A}tkin-{L}ehner.
\newblock {\em Comm. Algebra}, 36(6):2160--2170, 2008.

\bibitem{Magma}
W.~Bosma, J.~Cannon, and C.~Playoust.
\newblock The {M}agma algebra system. {I}. {T}he user language.
\newblock {\em J. Symbolic Comput.}, 24(3-4):235--265, 1997.
\newblock Computational algebra and number theory (London, 1993).

\bibitem{Box21b}
Josha Box.
\newblock Computing models for quotients of modular curves.
\newblock {\em Res. Number Theory}, 7(3):Paper No. 51, 34, 2021.

\bibitem{Box21a}
Josha Box.
\newblock Quadratic points on modular curves with infinite {M}ordell-{W}eil
  group.
\newblock {\em Math. Comp.}, 90(327):321--343, 2021.

\bibitem{BoxGajovicGoodman22}
Josha Box, Stevan Gajovi\'{c}, and Pip Goodman.
\newblock Cubic and {Q}uartic {P}oints on {M}odular {C}urves {U}sing
  {G}eneralised {S}ymmetric {C}habauty.
\newblock {\em Int. Math. Res. Not. IMRN}, 02 2022.

\bibitem{BruinStoll08}
Nils Bruin and Michael Stoll.
\newblock Deciding existence of rational points on curves: an experiment.
\newblock {\em Experiment. Math.}, 17(2):181--189, 2008.

\bibitem{Coleman85}
Robert~F. Coleman.
\newblock Torsion points on curves and {$p$}-adic abelian integrals.
\newblock {\em Ann. of Math. (2)}, 121(1):111--168, 1985.

\bibitem{DiamondShurman05}
Fred Diamond and Jerry Shurman.
\newblock {\em A first course in modular forms}, volume 228 of {\em Grad. Texts
  Math.}
\newblock Berlin: Springer, 2005.

\bibitem{GalbraithThesis}
Steven Galbraith.
\newblock {\em Equations for modular curves}.
\newblock PhD thesis, University of Oxford, 1996.
\newblock \url{https://www.math.auckland.ac.nz/~sgal018/thesis.pdf}.

\bibitem{KenkuMomose88}
M.~A. Kenku and Fumiyuki Momose.
\newblock Automorphism groups of the modular curves {$X_0(N)$}.
\newblock {\em Compositio Math.}, 65(1):51--80, 1988.

\bibitem{LMFDB}
The {LMFDB Collaboration}.
\newblock The {L}-functions and modular forms database.
\newblock \url{https://www.lmfdb.org}, 2024.
\newblock [Online; accessed 18 October 2024].

\bibitem{Siksek09b}
Samir Siksek.
\newblock Chabauty for symmetric powers of curves.
\newblock {\em Algebra Number Theory}, 3(2):209--236, 2009.

\bibitem{Siksekcnf}
Samir Siksek.
\newblock {Explicit Chabauty over number fields}.
\newblock {\em Algebra \& Number Theory}, 7(4):765 -- 793, 2013.

\bibitem{Stoll01}
Michael Stoll.
\newblock Implementing 2-descent for {J}acobians of hyperelliptic curves.
\newblock {\em Acta Arith.}, 98(3):245--277, 2001.

\bibitem{Sutherland12}
Andrew~V. Sutherland.
\newblock A local-global principle for rational isogenies of prime degree.
\newblock {\em J. Th\'{e}or. Nombres Bordeaux}, 24(2):475--485, 2012.

\bibitem{Zywina21}
David Zywina.
\newblock Computing actions on cusp forms.
\newblock 2021.
\newblock \url{https://arxiv.org/abs/2001.07270}.

\end{thebibliography}
\end{document}